%% file: 4.tex
\documentclass[12pt]{article}
\usepackage{graphicx}
\usepackage{amsmath}
\usepackage{amsfonts}
\usepackage{amsthm}
\usepackage[T1]{fontenc}
\usepackage{url}
\usepackage{color}
\usepackage[margin=1in]{geometry}
\input{mymacros}

\date{}
\begin{document}
\title{On the Korn interpolation and second inequalities in thin domains}
\author{D. Harutyunyan}
\maketitle

\begin{abstract}
We consider shells of non-constant thickness in three dimensional Euclidean space around surfaces which have bounded principal curvatures. We derive Korn's interpolation (or the so called first and a half\footnote{The inequality first introduced in [\ref{bib:Gra.Har.1}]}) and second inequalities on that kind of domains for $\Bu\in H^1$ vector fields, imposing no boundary or normalization conditions on $\Bu.$ The constants in the estimates are asymptotically optimal in terms of the domain thickness $h,$ with the leading order constant having the scaling $h$ as $h\to 0.$ This is the first work that determines the asymptotics of the optimal constant in the classical Korn second inequality for shells in terms of the domain thickness in almost full generality, the inequality being fulfilled for practically all thin domains $\Omega\in\mathbb R^3$ and all vector fields $\Bu\in H^1(\Omega).$ Moreover, the Korn interpolation inequality is stronger than Korn's second inequality, and it reduces the problem of estimating the gradient $\nabla\Bu$ in terms of the symmetrized gradient $e(\Bu)$, in particular any linear geometric rigidity estimates for thin domains, to the easier problem of proving the corresponding Poincar\'e-like estimates on the field $\Bu$ itself.
\end{abstract}

\section{Introduction}
\label{sec:1}
A shell of thickness $h$ in three dimensional Euclidean space is given by $\Omega=\{x+t\Bn(x) \ : \ x\in S,\ t\in [-h/2,h/2]\},$ where $S\subset\mathbb R^3$ is a bounded connected smooth enough regular surface with a unit normal $\Bn(x)$ at the point $x\in S.$ The surface $S$ is called the mid-surface of the shell $\Omega.$ Assume as above $S\subset\mathbb R^3$ is a bounded connected smooth enough regular surface with a unit normal $\Bn(x)$ at the point $x\in S.$ Let $h>0$ be a small parameter and assume the family of Lipschitz functions $g_1^h(x),g_2^h(x)\colon S\to (0,\infty)$ satisfy the uniform conditions
\begin{equation}
\label{1.1}
h\leq g_1^h(x),g_2^h(x)\leq c_1 h,\quad \text{and}\quad |\nabla g_1^h(x)|+|\nabla g_2^h(x)|\leq c_2h,\quad\text{for all}\quad x\in S.
\end{equation}
Then the set $\Omega^h=\{x+t\Bn(x) \ : \ x\in S,\ t\in [-g_1^h(x),g_2^h(x)]\}$ is a shell with thickness of order $h;$ which we will also call a thin domain. Understanding the rigidity of shells and more generally thin spatial domains is one of the challenges in nonlinear elasticity, where there are still many open questions. Unlike the situation for shells in general, the rigidity of plates has been quite well understood based on the work of Friesecke, James and M\"uller  [\ref{bib:Fri.Jam.Mue.1},\ref{bib:Fri.Jam.Mue.2}]. It is known that the rigidity of a shell $\Omega$ is closely related to the optimal Korn's constant in the nonlinear (in some cases linear) first Korn's inequality-a geometric rigidity estimate for $\Bu\in H^1(\Omega)$ fields [\ref{bib:Korn.1},\ref{bib:Korn.2},\ref{bib:Friedrichs},\ref{bib:Kohn},\ref{bib:Horgan},\ref{bib:Fri.Jam.Mue.1},\ref{bib:Fri.Jam.Mue.2}]. Depending on the problem, the field $\Bu\in H^1$ may or may not satisfy boundary conditions, e.g., [\ref{bib:Fri.Jam.Mue.1},\ref{bib:Gra.Tru.},\ref{bib:Gra.Har.2}]. It is also known that the critical buckling load of a shell under compression is again closely related to Korn's constant\footnote{The optimal constant in Korn's first inequality} in Korn's first inequality [\ref{bib:Gra.Tru.},\ref{bib:Gra.Har.2}], thus finding the optimal constants in Korn's inequalities is a central task in problems concerning thin domains in general. The geometric rigidity estimate in [\ref{bib:Fri.Jam.Mue.1}] reads as follows: \textit{Assume $\Omega\subset\mathbb R^3$ is open bounded connected and Lipschitz. Then there exists a constant $C_I=C_I(\Omega),$ such that for every vector field $\Bu\in H^1(\Omega),$ there exists a constant rotation $\BR\in SO(3)$, such that}
\begin{equation}
\label{1.2}
\|\nabla\Bu-\BR\|_{L^2(\Omega)}^2\leq C_{I}\int_\Omega\mathrm{dist}^2(\nabla\Bu(x),SO(3))dx.
\end{equation}
It is known that if $\Omega$ is a thin domain, then the constant $C_I$ in (\ref{1.2}) blows up as the the thickness of the domain goes to zero, in particular one has $C_I=\frac{C}{h^2}$ in the case of plates [\ref{bib:Fri.Jam.Mue.2}], see also the book of Tovstik and Smirnov [\ref{bib:Tov.Smi.}] on buckling of shells, where they construct various Ans\"atze for different shells giving the constant $C_I$ tending to infinity as $h$ goes to zero. Also, the estimate (\ref{1.2}) for plates is the cornerstone in the derivation of nonlinear plate theories [\ref{bib:Fri.Jam.Mue.2}] and shell theories in the case of low elastic energy [\ref{bib:Fri.Jam.Mor.Mue.}]. Let us point out, that at present there is no \textbf{nonlinear shell theory} derived from nonlinear three dimensional elasticity by $\Gamma-$convergence similar to the plate theory in [\ref{bib:Fri.Jam.Mue.2}], but there are works by different authors and groups in different directions for different situations [\ref{bib:Ciarlet3},\ref{bib:Ciarlet4},\ref{bib:Fri.Jam.Mor.Mue.},\ref{bib:LeD.Rao.1},\ref{bib:LeD.Rao.2},\ref{bib:Cia.Rab.},\ref{bib:Hor.Lew.Pak.},\ref{bib:Hor.Vel.},\ref{bib:Bla.Gri.1},\ref{bib:Bla.Gri.2},\ref{bib:Lew.Mor.Pak.},\ref{bib:Lew.Pak.},\ref{bib:Lew.Mah.Pak.},\ref{bib:Bel.Koh.1},\ref{bib:Bel.Koh.2},\ref{bib:Bel.Koh.3},\ref{bib:Bella},\ref{bib:Koh.Obr.}], some of them using already existing shell theories for certain situations. The main reason for the missing nonlinear shell theory is that the optimal constant asymptotics in the shell thickness $h$ in the estimate (\ref{1.2}) is an open problem for shells in general. Meanwhile, the lower bound (\ref{1.2}) with $C_I=Ch^{-2}$ is universal in the sense that it holds for shells too as shown in [\ref{bib:Fri.Jam.Mor.Mue.}], but it is not optimal if the shell has non-vanishing curvatures as shown in [\ref{bib:Gra.Har.3},\ref{bib:Harutyunyan.2}]. The linearization of (\ref{1.2}) around the identity matrix is Korn's first inequality [\ref{bib:Korn.1},\ref{bib:Korn.2},\ref{bib:Friedrichs},\ref{bib:Kon.Ole.1},\ref{bib:Kon.Ole.2},\ref{bib:Ciarlet1},\ref{bib:Cia.Mar.}] without boundary conditions and reads as follows:
\textit{Assume $\Omega\subset\mathbb R^n$ is open bounded connected and Lipschitz. Then there exists a constant $C_{II}=C(\Omega),$ depending only on $\Omega,$ such that for every vector field $\Bu\in H^1(\Omega)$ there exists a skew-symmetric matrix $\BA\in \mathbb R^{n\times n,}$ i.e., $A+A^T=0,$ such that
\begin{equation}
\label{1.3}
\|\nabla\Bu-\BA\|_{L^2(\Omega)}^2\leq C_{II}\|e(\Bu)\|_{L^2(\Omega)}^2,
\end{equation}
where $e(\Bu)=\frac{1}{2}(\nabla\Bu+\nabla\Bu^T)$ is the symmetrized gradient (the strain in linear elasticity).} There is also Korn's second inequality [\ref{bib:Korn.1},\ref{bib:Korn.2},\ref{bib:Kon.Ole.2}], which imposes no boundary or normalization condition on the vector field $\Bu\in H^1(\Omega)$ and reads a follows: \textit{Assume $\Omega\subset\mathbb R^n$ is open bounded connected and Lipschitz. Then there exists a constant $C=C(\Omega),$ depending only on $\Omega,$ such that for every vector field $\Bu\in H^1(\Omega)$ there holds:}
\begin{equation}
\label{1.4}
\|\nabla\Bu\|_{L^2(\Omega)}^2\leq C(\|\Bu\|_{L^2(\Omega)}^2+\|e(\Bu)\|_{L^2(\Omega)}^2).
\end{equation}
We refer to the recent survey by M\"uller [\ref{bib:Mueller.}] for more details on Korn and geometric rigidity estimates as well as plate and shell theories. It is natural to expect that the constants $C_I$ and $C_{II}$ have the same asymptotics in $h$, which is open for thin domains is general, however, in the meantime, the estimate $C_{I}\geq C_{II}$ for instance is quite straightforward. If $\Omega$ is a plate then the optimal constants in (\ref{1.2}) and (\ref{1.3}) scale like $h^{-2}$ [\ref{bib:Fri.Jam.Mue.2},\ref{bib:Gra.Har.3}]. When $\Omega\subset\mathbb R^3$ is a shell and the vector field $\Bu\in H^1(\Omega)$ satisfies zero or periodic Dirichlet boundary conditions on the thin face of the shell, then the following is known: In the case when one of the two principal curvatures of $S$ vanishes on the entire mid-surface $S$ and the other one has a constant sign, i.e., never vanishes, then the optimal constant $C_{II}$ in (\ref{1.3}) scales like $h^{-3/2}$ [\ref{bib:Gra.Har.3}]. If $S$ has a nonzero Gaussian curvature, then $C_{II}$ scales like $h^{-4/3}$ for the negative curvature and like $h^{-1}$ for the positive curvature [\ref{bib:Harutyunyan.2}]. Our program departed from [\ref{bib:Harutyunyan.1},\ref{bib:Gra.Har.3}] is to study the inequalities $(\ref{1.2})$ and (\ref{1.3}) for the remaining cases when the asymptotics of $C_{II}$ in $(\ref{1.3})$ has been established  [\ref{bib:Gra.Har.3},\ref{bib:Harutyunyan.2}]. In the present work we prove Korn's interpolation and second inequalities that hold and are sharp for practically all thin spatial domains $\Omega$ and all displacements $\Bu\in H^1(\Omega).$ The constants in the estimate are optimal and have the form $C$ or $Ch^{-1}.$ We also point out some immediate applications of the Korn interpolation inequality, in particular the estimates in [\ref{bib:Harutyunyan.1},\ref{bib:Gra.Har.3}] lead to new simplified sharp estimates for nonzero Gaussian curvature shells. The new interpolation estimate looks classical and it actually reduces the problem of proving (\ref{1.3}) to proving a Poincar\'e like inequality on the vector field $\Bu$ with $e(\Bu)$ in place of $\nabla\Bu.$

\section{Definitions and notation}
\setcounter{equation}{0}
\label{sec:2}

We introduce here the main notation and definitions. As in the case of shells with constant thickness, we will still call $S$ the mid-surface of the thin domain $\Omega.$ We will assume throughout this work that $S$ is connected, compact, regular and of class $C^3$ up to its boundary. We also assume that $S$ has a finite atlas of patches $S\subset\cup_{i=1}^k\Sigma_i$ such that each  patch $\Sigma_i$ can be parametrized by the principal variables $z$ and $\Gth$ ($z=$constant and $\Gth=$constant are the principal lines on $\Sigma_i$) that change in the ranges $z\in [z_i^1(\Gth),z_i^2(\Gth)]$ for $\Gth\in [0,\omega_i],$ where $\omega_i>0$ for $i=1,2,\dots,k.$ Moreover, the functions $z_i^1(\Gth)$ and $z_i^2(\Gth)$ satisfy the conditions
\begin{align}
\label{2.1}
&\min_{1\leq i\leq k}\inf_{\Gth\in [0,\omega_i]}[z_i^2(\Gth)-z_i^1(\Gth)]=l>0,\quad\max_{1\leq i\leq k}\sup_{\Gth\in [0,\omega_i]}[z_i^2(\Gth)-z_i^1(\Gth)]=L<\infty,\\ \nonumber
&\max_{1\leq i\leq k}\left(\|z_i^1\|_{W^{1,\infty}[0,\omega_i]}+\|z_i^2\|_{W^{1,\infty}[0,\omega_i]}\right)=Z<\infty.
\end{align}
Since there will be no condition imposed on the vector field $\Bu\in H^1(\Omega),$ (see Theorem~\ref{th:3.1}), we can restrict ourselves to a single patch $\Sigma_i\subset S$ and denote it by $S$ for simplicity. Let the mid-surface $S$ be given by the parametrization $\Br=\Br(\Gth,z)$ in the principal variables. Then denoting the normal coordinate by $t$ we obtain the set of local orthogonal curvilinear coordinates $(t,\Gth,z)$ on the entire domain $\Omega$ given by
$\BR(t,\Gth,z)=\Br(z,\Gth)+t\Bn(z,\Gth),$ where $\Bn$ is the unit normal to $S$ and $t\in [-g_1^h,g_2^h].$ 
Let
\[
A_{z}=\left|\frac{\partial \Br}{\partial z}\right|\quad \text{and}\quad A_{\Gth}=\left|\frac{\partial \Br}{\partial\Gth}\right|
\]
be the two nonzero components of the metric tensor of the mid-surface and let $\Gk_{z}$ and $\Gk_{\Gth}$ be the two principal curvatures.
In what follows we will use the notation $f_{,\alpha}$ for the partial derivative $\frac{\partial}{\partial\alpha}$ inside the gradient matrix of a vector field $\Bu\colon\Omega\to\mathbb R^3.$ For the partial derivatives in the gradient of vector fields $\BU=(u,v)\colon E\to\mathbb R^2,$ i.e., the two dimensional ones, we will use the simplified notation $u_\alpha,$ where $E\subset\mathbb R^2$ is any open subset of $\mathbb R^2.$  The gradient of a vector field $\Bu=(u_t,u_\Gth,u_z)\in H^1(\Omega,\mathbb R^3)$ on the entire set $\Omega$ is given by the formula
\begin{equation}
\label{2.2}
\nabla\Bu=
\begin{bmatrix}
  u_{t,t} & \dfrac{u_{t,\Gth}-A_{\Gth}\Gk_{\Gth}u_{\Gth}}{A_{\Gth}(1+t\Gk_{\Gth})} &
\dfrac{u_{t,z}-A_{z}\Gk_{z}u_{z}}{A_{z}(1+t\Gk_{z})}\\[3ex]
u_{\Gth,t}  &
\dfrac{A_{z}u_{\Gth,\Gth}+A_{z}A_{\Gth}\Gk_{\Gth}u_{t}+A_{\Gth,z}u_{z}}{A_{z}A_{\Gth}(1+t\Gk_{\Gth})} &
\dfrac{A_{\Gth}u_{\Gth,z}-A_{z,\Gth}u_{z}}{A_{z}A_{\Gth}(1+t\Gk_{z})}\\[3ex]
u_{ z,t}  & \dfrac{A_{z}u_{z,\Gth}-A_{\Gth,z}u_{\Gth}}{A_{z}A_{\Gth}(1+t\Gk_{\Gth})} &
\dfrac{A_{\Gth}u_{z,z}+A_{z}A_{\Gth}\Gk_{z}u_{t}+A_{z,\Gth}u_{\Gth}}{A_{z}A_{\Gth}(1+t\Gk_{z})}
\end{bmatrix}
\end{equation}
in the orthonormal local basis $(\Bn,\Be_\Gth,\Be_z).$ The gradient restricted to the mid-surface or the so called simplified gradient denoted by $\BF$ is obtained from (\ref{2.2}) by putting $t=0,$
thus it has the form
\begin{equation}
\label{2.3}
\BF=
\begin{bmatrix}
  u_{t,t} & \dfrac{u_{t,\Gth}-A_{\Gth}\Gk_{\Gth}u_{\Gth}}{A_{\Gth}} &
\dfrac{u_{t,z}-A_{z}\Gk_{z}u_{z}}{A_{z}}\\[3ex]
u_{\Gth,t}  &
\dfrac{A_{z}u_{\Gth,\Gth}+A_{z}A_{\Gth}\Gk_{\Gth}u_{t}+A_{\Gth,z}u_{z}}{A_{z}A_{\Gth}} &
\dfrac{A_{\Gth}u_{\Gth,z}-A_{z,\Gth}u_{z}}{A_{z}A_{\Gth}}\\[3ex]
u_{ z,t}  & \dfrac{A_{z}u_{z,\Gth}-A_{\Gth,z}u_{\Gth}}{A_{z}A_{\Gth}} &
\dfrac{A_{\Gth}u_{z,z}+A_{z}A_{\Gth}\Gk_{z}u_{t}+A_{z,\Gth}u_{\Gth}}{A_{z}A_{\Gth}}
\end{bmatrix}.
\end{equation}
We will work with $\BF$ and then pass to $\nabla\Bu$ using their closeness to the order of $h$ due to the smallness of the variable $t.$ In this paper all norms $\|\cdot\|$ are $L^{2}$ norms and the Cartesian $L^2$ inner product of two functions
$f,g\colon\Omega\to\mathbb R$ will be given by
\[
(f,g)_{\Omega}=\int_{\Omega}A_zA_\Gth f(t,\Gth,z)g(t,\Gth,z)d\Gth dzdt,
\]
which gives rise to the norm $\|f\|_{L^2(\Omega)}$. The thin domain mid-surface parameters are the quantities $\omega,l,L,Z,a,A$ and $k$
that will be (some of them defined below) assumed to satisfy the below conditions
\begin{align}
\label{2.4}
&\omega,l,L,Z,a=\min_{D}(A_\Gth,A_z)>0, \quad A=\|A_\Gth\|_{W^{2,\infty}(D)}+\|A_z\|_{W^{2,\infty}(D)}<\infty,\\ \nonumber
&k=\|\Gk_\Gth\|_{W^{1,\infty}(D)}+\|\Gk_z\|_{W^{1,\infty}(D)}<\infty,
\end{align}
where $D=\{(\Gth,z)\ : \ \Gth\in [0,\omega], z\in[z^1(\Gth),z^2(\Gth)]\}.$ 

\section{Main results}
\label{sec:3}
\setcounter{equation}{0}

Let us point out that in what follows the constants $h_0>0$ and $C>0$ will depend only on the quantities $a,A,k,\omega,l,L,Z,c_1$ and $c_2$ i.e., the domain mid-surface and thickness parameters. 
\begin{theorem}[Korn's interpolation inequality]
\label{th:3.1}
Assume the conditions (\ref{2.1}) and (\ref{2.4}) hold. There exists constants $h_0,C>0,$ such that Korn's interpolation inequality holds:
\begin{equation}
  \label{3.1}
\|\nabla\Bu\|^2\leq C\left(\frac{\|\Bu\cdot\Bn\|\cdot\|e(\Bu)\|}{h}+\|\Bu\|^2+\|e(\Bu)\|^2\right),
\end{equation}
for all $h\in(0,h_0)$ and $\Bu=(u_t,u_\Gth,u_z)\in H^1(\Omega),$ where $\Bn$ is the unit normal to the mid-surface $S.$ Moreover, the exponent of $h$ in the inequality (\ref{3.1}) is optimal for any thin domain $\Omega$ satisfying (\ref{2.1}) and (\ref{2.4}), i.e., there exists a displacement $\Bu\in H^1(\Omega,\mathbb R^3)$ realizing the asymptotics of $h$ in (\ref{3.1}).
\end{theorem}

The next theorem provides a sharp Korn's second inequality for thin domains.

\begin{theorem}[Korn's second inequality]
\label{th:3.2}
Assume the conditions (\ref{2.1}) and (\ref{2.4}) hold. There exists constants $h_0,C>0,$ such that Korn's second inequality holds:
\begin{equation}
  \label{3.2}
\|\nabla\Bu\|^2\leq \frac{C}{h}(\|\Bu\|^2+\|e(\Bu)\|^2),
\end{equation}
for all $h\in(0,h_0)$ and $\Bu=(u_t,u_\Gth,u_z)\in H^1(\Omega).$ Moreover, the exponent of $h$ in the inequality (\ref{3.2}) is optimal for any thin spatial domain $\Omega$ satisfying (\ref{2.1}) and (\ref{2.4}), i.e., there exists a displacement $\Bu\in H^1(\Omega,\mathbb R^3)$ realizing the asymptotics of $h$ in (\ref{3.2}).
\end{theorem}

Some remarks are in order. Some little manipulations of the estimate (\ref{3.1}) have been proven in [\ref{bib:Gra.Har.3}] for shells with zero Gaussian curvature and in [\ref{bib:Harutyunyan.2}] for shells with nonzero Gaussian curvature. But in both works the mid-surface $S$ was required to be a single patch and the vector field $\Bu$ was required to satisfy zero (or partially periodic) boundary conditions on the thin face of $\Omega.$ The strength of (\ref{3.1}) and (\ref{3.2})
is that they hold for practically all shells $\Omega$ and absolutely all vector fields $\Bu\in H^1(\Omega).$ 

Another very important remark is as follows: it is already clear from the works [\ref{bib:Gra.Har.3},\ref{bib:Harutyunyan.2}] that when trying to prove (\ref{1.3}) or (\ref{1.2}) for shells, one needs to consider the cases 1) $\Gk_\Gth=0$ and $|\Gk_z|>0,$ 2) $\Gk_\Gth\Gk_z>0,$ 3) $\Gk_\Gth\Gk_z<0.$ While the Ansatz given at the end of Section~5 realizes the asymptotics of both (\ref{3.1}) and (\ref{3.2}), it works only in the case 2) for the purpose of the estimate (\ref{1.3}). In the cases 1) and 2) the Ans\"atze constructed in [\ref{bib:Gra.Har.3}] and [\ref{bib:Harutyunyan.2}] do work for the purpose of (\ref{1.3}), which also work for (\ref{3.1}). This being said, while Korn's second inequality is classical, the interpolation inequality (\ref{3.1}) seems to be the "best" asymptotic Korn second-like inequality holding true and being sharp for all main shell-curvature situations.

\section{The key lemma}
\label{sec:4}
\setcounter{equation}{0}
In this section we prove a gradient separation estimate for harmonic functions in two dimensional thin domains.
We start with the following rigidity estimate:
\begin{lemma}
\label{lem:4.1}
Let $h\in (0,1)$ and let the Lipschitz functions $\varphi_1,\varphi_2\colon[0,1]\to (0,\infty)$ and the constants $C_1,C_2>0$ be such that 
\begin{equation}
\label{4.1}
h\leq \varphi_i(y)\leq C_1h,\quad |\nabla\varphi_i(y)|\leq C_2h \quad\text{for all}\quad y\in [0,1], \ \ i=1,2.
\end{equation}
Denote the thin domain $D=\{(x,y)\in\mathbb R^2 \ : \ y\in (0,1), x\in (-\varphi_1(y),\varphi_2(y))\}$ that has a thickness of order $h.$ Then there exists a constant $c>0,$ depending only on $C_1$ and $C_2,$ such that any harmonic function $w\in C^2(D)$ fulfills the inequality
\begin{equation}
\label{4.2}
\|w_y-a\|_{L^2(D)}\leq\frac{c}{h}\|w_x\|_{L^2(D)},
\end{equation}
where $a=\frac{1}{|D|}\int_D w_y$ is the average of $w_y$ over $D.$
\end{lemma}

\begin{proof}
The proof is derived from Korn's first inequality on $D$ for an appropriately chosen vector field $\BU\in H^1(D).$ It is divided into two steps.\\ 
\textbf{Step 1. Korn's first inequality on $D.$} \textit{There exists a constant $C>0,$ depending only on $C_1$ and $C_2$ such that for any vector field $\BU=(u,v)\colon D\to\mathbb R^2,$ there exists a skew-symmetric matrix $A,$ such that}
\begin{equation}
\label{4.3}
\|\nabla\BU-\BA\|_{L^2(D)}\leq \frac{C}{h}\|e(\BU)\|_{L^2(D)}.
\end{equation}
In the proof of (\ref{4.3}) the constant $C>0$ may depend only on $C_1$ and $C_2.$ We adopt the localization argument of Kohn and Vogelius  [\ref{bib:Koh.Vog.}] that was also successfully employed in [\ref{bib:Fri.Jam.Mue.1},\ref{bib:Fri.Jam.Mue.2}]. Denote the positive whole number 
$N=\left[\frac{1}{h}\right]+1$ where $[z]$ is the whole part of the real number $z\in\mathbb R.$ Consider the domains $D_k=\left\{(x,y)\in\mathbb R^2 \ : \ y\in \left(\frac{k-1}{N},\frac{k+1}{N}\right), x\in (-\varphi_1(y),\varphi_2(y))\right\},$ for $k=1,2,\dots,N-1.$ The obvious estimates $\frac{1}{h}<N<\frac{2}{h}$ together with (\ref{4.1}) ensure that each of the rescaled domains $\frac{1}{h}\cdot D_k$ is of order one with a piecewise Lipschitz boundary. Thus Korn's first inequality\footnote{Although it is not stated in [\ref{bib:Kon.Ole.2}], it can be easily shown that the constant $C$ in Theorem~6 depends only on $C_1$ and $C_2$ for our situation.} [\ref{bib:Kon.Ole.2}, Theorem~6] provides the existence of the sequence of skew-symmetric matrices 
$\BA_1,\BA_2,\dots,\BA_{N-1}$ such that we have for the rescaled fields $\tilde\BU(x,y)=\BU\left(\frac{x}{h},\frac{y}{h}\right),$
\begin{equation}
\label{4.4}
\|\nabla\tilde \BU-\BA_k\|_{L^2(\frac{1}{h}\cdot D_k)}\leq C\|e(\tilde \BU)\|_{L^2(\frac{1}{h}\cdot D_k)},\quad k=1,2,\dots,N-1.
\end{equation}
If we undo the rescaling in (\ref{4.4}) we get the same inequalities for the original field $\BU$ in $D_k:$
\begin{equation}
\label{4.5}
\|\nabla \BU-\BA_k\|_{L^2(D_k)}\leq C\|e(\BU)\|_{L^2(D_k)},\quad k=1,2,\dots,N-1.
\end{equation}
The rest is standard by the estimation of the norm $\|\BA_i-\BA_j\|.$ Assume $N>2$ otherwise we are done. For $1\leq k \leq N-2$ we have by the triangle inequality and by (\ref{4.5}) the estimate 
\begin{align*}
\|\BA_k-\BA_{k+1}\|_{L^2(D_k\cap D_{k+1})}&\leq \|\nabla \BU-\BA_k\|_{L^2(D_k)} +\|\nabla \BU-\BA_{k+1}\|_{L^2(D_{k+1})}\\
&\leq C\|e(\BU)\|_{L^2(D_k)}+C\|e(\BU)\|_{L^2(D_{k+1})},
\end{align*}
thus taking into account the fact that the measure of $D_k\cap D_{k+1}$ is of order $h^2$ we get from the last inequality the estimate 
\begin{equation}
\label{4.6}
|\BA_k-\BA_{k+1}|\leq \frac{C}{h}(\|e(\BU)\|_{L^2(D_k)}+\|e(\BU)\|_{L^2(D_{k+1})}), \quad k=1,2,\dots,N-2.
\end{equation}
Now an application of (\ref{4.6}) and the triangle inequality implies for any index $1\leq k\leq N-1$ the estimate 
\begin{align}
\label{4.7}
\|\BA_1-\BA_k\|_{L^2(D_k)} & \leq Ch |\BA_1-\BA_k| \\ \nonumber
& \leq Ch\sum_{i=1}^{k-1}|\BA_{i}-\BA_{i+1}| \\ \nonumber
& \leq C \sum_{i=1}^{k-1} (\|e(\BU)\|_{L^2(D_i)}+\|e(\BU)\|_{L^2(D_{i+1})}) \\ \nonumber
& \leq 2C\|e(\BU)\|_{L^2(D)}.
\end{align}
Therefore choosing $\BA=\BA_1,$ we get taking into account the estimates (\ref{4.7}), (\ref{4.5}) and the triangle inequality,
\begin{align*}
\|\nabla\BU-\BA\|_{L^2(D)} & \leq \sum_{k=1}^{N-1} \left(\|\nabla\BU-\BA_k\|_{L^2(D_k)}+\|\BA-\BA_k\|_{L^2(D_k)}\right)\\
& \leq \sum_{i=1}^{k-1}C\left\|e(\BU)\|_{L^2(D_k)}+ e(\BU)\|_{L^2(D)}\right)\\
& \leq CN\|e(\BU)\|_{L^2(D)}\\
& \leq \frac{C}{h}\|e(\BU)\|_{L^2(D)},\\
\end{align*}
which is exactly (\ref{4.3}).\\
\textbf{Step 2. Proof of Lemma~\ref{lem:4.1}.} We make the specific choice 
\begin{equation}
\label{4.8}
u(x,y)=w(x,y),\quad\text{ and}\quad v(x,y)=-\int_0^x w_y(t,y)dt+\int_0^y w_x(0,z)dz,
\end{equation}
where using the harmonicity of $w$ we can calculate
$$u_x=w_x,\quad u_y=w_y,\quad v_x=-w_y(x,y),\quad v_y=w_x(x,y),$$
thus we obtain 
\begin{equation}
\label{4.9}
\nabla\BU=
\begin{bmatrix}
w_x & w_y\\
-w_y & w_x
\end{bmatrix},\quad
e(\BU)=
\begin{bmatrix}
w_x & 0\\
0 & w_x
\end{bmatrix}.
\end{equation}
Combining (\ref{4.3}) and (\ref{4.9}) we establish (\ref{4.2}) with $a_{12}$ (where $a_{12}$ is one of the off-diagonal entry of the skew-symmetric matrix $\BA$ coming from the estimate (\ref{1.3}) for the vector field $\BU$ in the domain $
D$) instead of $a,$ but it is clear that the quantity $\|w_y-\mu\|_{L^2(D)}^2$ is minimized at $\mu=a,$ thus we get (\ref{4.2}).
\end{proof}

The next lemma provides a key estimate in the analysis.
\begin{lemma}
\label{lem:4.2}
Let $h\in\left(0,\frac{1}{5}\right)$ and let $\varphi_1,\varphi_2,C_1,C_2$ and the domain $D$ be as in Lemma~\ref{lem:4.1}. There exists a constant $C>0,$ depending  only on $C_1$ and $C_2,$ such that any harmonic function $w\in C^2(D)$ fulfills the inequality
\begin{equation}
\label{4.10}
\|w_y\|_{L^2(D)}^2\leq C\left(\frac{1}{h}\|w\|_{L^2(D)}\cdot\|w_x\|_{L^2(D)}+\|w\|_{L^2(D)}^2+\|w_x\|_{L^2(D)}^2\right).
\end{equation}
\end{lemma}

\begin{proof}
We again divide the proof into three steps for the convenience of the reader. In the first step we obtain an interior estimate on the partial derivative $w_y,$ in the second step we prove a similar estimate on $w_y$ in the parts of the domain that are located close to the horizontal boundary of $D$,  and in the third step we extend the interior estimate from the first part up to the boundary and thus complete the proof. Let us point out that all the norms in the proof are $L^2(D)$ and the constant $C>0$ depends only on $C_1$ and $C_2$ unless specified otherwise.\\
\textbf{Step 1. An interior estimate on $w_y.$} \textit{There exists a constant $C>0$ such that for any harmonic function $w\in C^2(R)$ the inequality holds:}
\begin{equation}
\label{4.11}
\int_{(-\frac{h}{2},\frac{h}{2})\times(0,1)}|w_y|^2\leq C\left(\frac{1}{h}\|w\|\cdot\|w_x\|+\|w\|^2+\|w_x\|^2\right).
\end{equation}
Denote for $t\in [0,h]$ and $z\in [0,\frac{1}{2}]$ the rectangles $R_{t,z}=(-t,t)\times(z,1-z),$ $R_{z}^{top}=(-h,h)\times(1-z,1),$ 
$R_{z}^{bot}=(-h,h)\times(0,z)$ and the middle, top and bottom parts of $D$ similarly:  
$D_{z}=\{(x,y) \ : \ y\in(z,1-z), x\in (-\varphi_1(y),\varphi_2(y))\},$  $D_{z}^{top}=\{(x,y) \ : \ y\in1-z,1), x\in (-\varphi_1(y),\varphi_2(y))\},$ 
$D_{z}^{bot}=\{(x,y) \ : \ y\in(0,z), x\in (-\varphi_1(y),\varphi_2(y))\}.$ 
Let now $z\in(h,\frac{1}{4}]$ be a parameter and let $\varphi(y)\colon[0,1]\to[0,1]$ be a smooth cutoff function such that
$\varphi(y)=1$ for $y\in[z,1-z]$ and  $|\nabla \varphi(y)|\leq \frac{2}{z}$ for $y\in[0,1].$ We multiply the equality
$-\Delta w=0$ in $D$ by $\varphi w$ and integrate the obtained identity by parts over $R_{t,0}$ to get
$$\int_{R_{t,0}}\nabla(\varphi w)\cdot\nabla w=\int_0^1 \left([\varphi ww_x](t,y)-[\varphi ww_x](-t,y)\right)dy,$$
which due to the choice of $\varphi$ implies the estimate
\begin{equation}
\label{4.12}
\int_{R_{t,z}}|\nabla w|^2\leq \int_0^1 \left(|[\varphi ww_x](t,y)|+|[\varphi ww_x](t,y)|\right)dy
+\frac{2}{z}\int_{R_{z}^{bot}\cup R_{z}^{top}}|ww_y|.
\end{equation}
Integrating (\ref{4.12}) in $t$ over $(\frac{h}{2},h)$ we discover
\begin{equation*}
\int_{R_{\frac{h}{2},z}}|\nabla w|^2\leq \frac{2}{h}\int_{D}|ww_x|+\frac{2}{z}\int_{R_{z}^{bot}\cup R_{z}^{top}}|ww_y|,
\end{equation*}
which in turn implies the estimate (by the Cauchy-Schwartz)
\begin{equation}
\label{4.13}
\int_{R_{\frac{h}{2},z}}|\nabla w|^2\leq \frac{2}{h}\int_{D}|ww_x|+\frac{1}{\epsilon^2z^2}\int_{R_{z}^{bot}\cup R_{z}^{top}}w^2
+\epsilon^2\int_{R_{z}^{bot}\cup R_{z}^{top}}|w_y|^2,
\end{equation}
where $\epsilon>0$ is a parameter yet to be chosen. It is clear that (\ref{4.13}) gives for $2z$ in place of $z$ the estimate
\begin{equation}
\label{4.14}
\int_{R_{\frac{h}{2},2z}}|\nabla w|^2\leq \frac{2}{h}\int_{D}|ww_x|+\frac{1}{4\epsilon^2z^2}\int_{R_{2z}^{bot}\cup R_{2z}^{top}}|w|^2
+\epsilon^2\int_{R_{2z}^{bot}\cup R_{2z}^{top}}|w_y|^2.
\end{equation}
Note, that the estimate (\ref{4.3}) is invariant under variable change $(x,y)\to (\lambda x,\lambda y),$ thus we have for the function $w$ on the domains 
$D_{2z}^{bot}$ and $D_{2z}^{top}:$ 
\begin{equation}
\label{4.15}
\int_{D_{2z}^{bot}} |w_y-a_1|^2\leq \frac{cz^2}{h^2}\int_{D_{2z}^{bot}} |w_x|^2,\quad\text{and}\quad
\int_{D_{2z}^{top}} |w_y-a_2|^2\leq \frac{cz^2}{h^2}\int_{D_{2z}^{top}} |w_x|^2,
\end{equation}
for some $a_1,a_2\in\mathbb R.$ We next have from (\ref{4.15}) and the triangle inequality, that
\begin{align}
\label{4.16}
\int_{R_{\frac{h}{2},z}}|\nabla w|^2&\geq \int_{R_{\frac{h}{2},z}}|w_y|^2\\ \nonumber
&\geq \int_{(-\frac{h}{2},\frac{h}{2})\times(z,2z)\cup(-\frac{h}{2},\frac{h}{2})(1-2z,1-z)}|w_y|^2\\ \nonumber
&\geq \frac{1}{2}\int_{(-\frac{h}{2},\frac{h}{2})\times(z,2z)}a_1^2+\frac{1}{2}\int_{(-\frac{h}{2},\frac{h}{2})\times(1-2z,1-z)}a_2^2\\ \nonumber
&-\int_{(-\frac{h}{2},\frac{h}{2})\times(z,2z)}|w_y-a_1|^2-\int_{(-\frac{h}{2},\frac{h}{2})\times(1-2z,1-z)}|w_y-a_2|^2\\ \nonumber
&\geq \frac{hz}{2}(a_1^2+a_2^2)-\frac{cz^2}{h^2}\int_{D_{2z}^{bot}} |w_x|^2-\frac{cz^2}{h^2}\int_{D_{2z}^{top}}|w_x|^2.
\end{align}
We similarly have the estimates
\begin{align}
\label{4.17}
\epsilon^2\int_{D_{z}^{bot}}|w_y|^2&\leq 2\epsilon^2\int_{D_{z}^{bot}}|w_y-a_1|^2+2\epsilon^2\int_{D_{z}^{bot}}a_1^2\\ \nonumber
&\leq \frac{2c\epsilon^2z^2}{h^2}\int_{D_{z}^{bot}} |w_x|^2+C\epsilon^2zha_1^2,
\end{align}
and
\begin{align}
\label{4.18}
\epsilon^2\int_{D_{z}^{top}}|w_y|^2&\leq 2\epsilon^2\int_{D_{z}^{top}}|w_y-a_2|^2+2\epsilon^2\int_{D_{z}^{top}}a_2^2\\ \nonumber
&\leq \frac{2c\epsilon^2z^2}{h^2}\int_{D_{z}^{top}} |w_x|^2+C\epsilon^2zha_2^2.
\end{align}
Combining now (\ref{4.13}) and (\ref{4.16})-(\ref{4.18}) we discover
$$
hz\left(\frac{1}{2}-C\epsilon^2\right)(a_1^2+a_2^2)\leq  \frac{4}{h}\int_{D}|ww_x|+\frac{1}{\epsilon^2z^2}\int_{D_{z}^{bot}\cup D_{z}^{top}}|w|^2
+\frac{cz^2}{h^2}\int_{D_{2z}^{bot}\cup D_{2z}^{top}}|w_x|^2+\frac{2c\epsilon^2z^2}{h^2}\int_{D_{z}^{bot}\cup D_{z}^{top}}|w_x|^2,
$$
which gives by choosing $\epsilon=\frac{1}{2\sqrt{C}},$
\begin{equation}
\label{4.19}
\frac{hz}{4}(a_1^2+a_2^2)\leq \frac{4}{h}\int_{D}|ww_x|+\frac{C}{z^2}\int_{D_{z}^{bot}\cup D_{z}^{top}}|w|^2
+\frac{Cz^2}{h^2}\int_{D_{2z}^{bot}\cup D_{2z}^{top}}|w_x|^2.
\end{equation}
Next we combine (\ref{4.15}) and (\ref{4.19}), and apply the triangle inequality to get the bound
\begin{equation}
\label{4.20}
\int_{D_{2z}^{bot}\cup D_{2z}^{top}}|w_y|^2\leq C\left(\frac{1}{h}\int_{D}|ww_x|+\frac{1}{z^2}\int_{D_{z}^{bot}\cup D_{z}^{top}}|w|^2
+\frac{z^2}{h^2}\int_{D_{2z}^{bot}\cup D_{2z}^{top}}|w_x|^2\right).
\end{equation}
Consequently we get from (\ref{4.14}) (for $\epsilon=1$) and (\ref{4.20}) the key interior estimate
\begin{equation}
\label{4.21}
\int_{R_{\frac{h}{2},0}}|w_y|^2\leq C\left(\frac{1}{h}\int_{D}|ww_x|+\frac{1}{z^2}\|w\|^2
+\frac{z^2}{h^2}\|w_x\|^2\right).
\end{equation}
The strategy is to minimize the right hand side of (\ref{4.21}) subject to the constraint $h\leq z\leq\frac{1}{4}$ on the parameter $z.$ Denote $z_0=\left(\frac{h\|w\|}{\|w_x\|}\right)^{1/2}$ and consider the following cases:\\
\textbf{Case 1.} Assume $\|w\|=0.$ In this case (\ref{4.11}) is obviously fulfilled.\\
\textbf{Case 2.} Assume $\|w_x\|=0.$ In this case we get (\ref{4.11}) from (\ref{4.21}) by choosing $z=\frac{1}{5}.$\\
\textbf{Case 3.} Assume
\begin{equation}
\label{4.22}
\|w\|,\|w_x\|>0\quad\text{and}\quad z_0\in \left[h,\frac{1}{4}\right).
\end{equation}
In this case to optimize (\ref{4.21}), one must naturally choose $z$ so that $\frac{1}{z^2}\|w\|^2=\frac{z^2}{h^2}\|w_x\|^2,$ which gives the value $z=z_0,$ that meets the constraint $h\leq z<\frac{1}{4}$ due to the assumptions in (\ref{4.22}). The value of $\frac{1}{z^2}\|w\|^2+\frac{z^2}{h^2}\|w_x\|^2$ with the above choice will be $\frac{2}{h}\|w\|\cdot\|w_x\|$ and (\ref{4.11}) will follow from (\ref{4.21}) by the Cauchy-Schwartz inequality.\\
\textbf{Case 4.} Assume
\begin{equation}
\label{4.23}
\|w\|,\|w_x\|>0\quad\text{and}\quad z_0>\frac{1}{4}.
\end{equation}
In this case the choice of a $z$ is again straightforward and we make the choice $z=\frac{1}{4}.$ It is then clear that we have by virtue of (\ref{4.23}), the estimate
$$\frac{z^2}{h^2}\|w_x\|^2=\frac{1}{16h^2}\|w_x\|^2\leq \|w\|^2,$$
thus (\ref{4.11}) follows.\\
\textbf{Case 5.} Assume
\begin{equation}
\label{4.24}
\|w\|,\|w_x\|>0\quad\text{and}\quad z_0<h.
\end{equation}
The choice of $z$ in this case is again quite straightforward, which is actually $z=h.$ The condition $z_0<h$ gives the estimate
$\frac{1}{h^2}\|w\|^2\leq \|w_x\|^2,$ thus (\ref{4.11}) again follows from (\ref{4.21}).\\
\textbf{Step 2. An estimate near the horizontal boundary of $D$.} \textit{There exists a constant $C>0,$ such that for any harmonic function $w\in C^2(R)$ the inequality holds:}
\begin{equation}
\label{4.25}
\int_{D_{h}^{bot}\cup D_{h}^{top}}|w_y|^2\leq C\left(\frac{1}{h}\|w\|\cdot\|w_x\|+\|w\|^2+\|w_x\|^2\right).
\end{equation}
The strategy of proving (\ref{4.25}) is to obtain it from (\ref{4.20}) by a suitable choice of $z$ subject to the constraint $h\leq z\leq\frac{1}{4}.$ The proof is the same as above, including the choice of $z$ thus we omit the details here. \\
\textbf{Step 3. Proof of (\ref{4.10}).} We recall the following two auxiliary lemmas proven in [\ref{bib:Harutyunyan.1}, Lemma~2.4] and in [\ref{bib:Kon.Ole.2}, Lemma~3] respectively.  
\begin{lemma}
\label{lem:4.3}
Assume $\lambda\in (0,1)$, $0<a<b$ and $f\colon[a,b]\to\mathbb R$ is absolutely continuous. Then the inequality holds:
\begin{equation}
\label{4.26}
\int_{a+\lambda(b-a)}^b f^2(t)dt\leq \frac{2}{\lambda}\int_{a}^{a+\lambda(b-a)}f^2(t)dt+4\int_a^{b}(b-t)^2f'^2(t)dt.
\end{equation}
\end{lemma}
\begin{lemma}
\label{lem:4.4}
Let $n\in\mathbb N,$ and let $\Omega\subset\mathbb R^n$ be open bounded and connected. Denote $\delta(x)=\mathrm{dist}(x,\partial\Omega).$ Then for any harmonic function $u\in C^2(\Omega)$ there holds:
\begin{equation}
\label{4.27}
\|\delta\nabla u\|_{L^2(\Omega)}\leq 2\|\nabla u\|_{L^2(\Omega)}.
\end{equation}
\end{lemma}
We fix a point $y\in (h,1-h)$ and apply Lemma~\ref{lem:4.3} to the function $w_y(x,y)$ on the segment $[0,\varphi_2(y)]$ as a function in $x$ for the value 
$\lambda=\frac{h}{2\varphi_2(y)}.$ We have by virtue of (\ref{4.1}), that
\begin{align*}
\int_{\frac{h}{2}}^{\varphi_2(y)}|w_y(x,y)|^2dx\leq & \frac{4\varphi_2(y)}{h}\int_{0}^{\frac{h}{2}}|w_y(x,y)|^2dx+4\int_{0}^{\varphi_2(y)}|(\varphi_2(y)-x)w_{xy}(x,y)|^2dx\\
&\leq 4C_1\int_ 0^{\frac{h}{2}}|w_y(x,y)|^2dx+4\int_{0}^{\varphi_2(y)}|(\varphi_2(y)-x)w_{xy}(x,y)|^2dx,
\end{align*}
thus integrating in $y$ over $(h,1-h)$ we obtain the estimate
\begin{equation}
\label{4.28}
\int_{T}|w_y|^2\leq 4C_1\int_{(0,h/2)\times(h,1-h)}|w_y|^2+4\int_{T\cup \left((0,\frac{h}{2})\times(h,1-h)\right)}|(\varphi_2(y)-x)w_{xy}|^2,
\end{equation}
where we set $T=\{(x,y) \ : \ y\in(h,1-h), x\in(\frac{h}{2},\varphi_2(y))\}.$ Observe that $w$ being harmonic in $D$ is smooth and thus $w_x$ is harmonic in $D$ as well. On the other hand due to the bounds (\ref{4.1}) we have $|\varphi_2(y)-x|\leq C\delta(x,y),$ where $\delta(x,y)$ is the distance function from the boundary of $D.$ Therefore we get by Lemma~\ref{lem:4.4} the estimate
$$
\int_{T\cup \left((0,h/2)\times(h,1-h)\right)}|(\varphi_2(y)-x)w_{xy}|^2\leq C\int_{D}| \delta \nabla w_x|^2 \leq C \int_{D}|w_x|^2,
$$
which gives owing back to (\ref{4.28}) the key estimate
\begin{equation}
\label{4.29}
\int_{T}|w_y|^2\leq C\int_{(0,h/2)\times(h,1-h)}|w_y|^2+C\int_{D}|w_{x}|^2.
\end{equation}
Similarly we have the same estimate for the left part of the rectangle:
\begin{equation}
\label{4.30}
\int_{T'}|w_y|^2\leq C\int_{(-\frac{h}{2},0)\times(h,1-h)}|w_y|^2+C\int_{D}|w_{x}|^2,
\end{equation}
where $T'=\{(x,y) \ : \ y\in(h,1-h), x\in(-\varphi_1(y), -\frac{h}{2})\}.$ It remains to combine the estimates (\ref{4.11}), (\ref{4.25}), (\ref{4.29}) and (\ref{4.30}) to obtain (\ref{4.10}). 
\end{proof}
The following consequence of Lemma~\ref{lem:4.2} will be useful in the proof of the main results. 
\begin{lemma}
\label{lem:4.5}
Let $h,b>0$ be such that $h<\frac{b}{4}.$ Assume $\varphi_1,\varphi_2 \colon [0,b]\to (h,\infty)$ are Lipschitz such that 
\begin{equation}
\label{4.31}
h\leq \varphi_i(y)\leq C_1h,\quad\text{and}\quad |\nabla\varphi_i(y)|\leq C_2h\quad \text{for all}\quad y\in [0,b], i=1,2.
\end{equation}
Set $D=\{(x,y) \ : \ y\in (0,b), x\in (-\varphi_1(y), \varphi_2(y))\}.$ Then there exists a constant $C>0,$ depending 
only on $C_1$ and $C_2,$ such that any harmonic function $w\in C^2(D)$ fulfills the inequality
\begin{equation}
\label{4.32}
\|w_y\|_{L^2(D)}^2\leq C\left(\frac{1}{h}\|w\|_{L^2(D)}\cdot\|w_x\|_{L^2(D)}+\frac{1}{b^2}\|w\|_{L^2(D)}^2+\|w_x\|_{L^2(D)}^2\right).
\end{equation}
\end{lemma}
\begin{proof}
The proof follows from Lemma~\ref{lem:4.2} by applying it to the harmonic function $v(x,y)=w(bx,by)$ defined on the rescaled domain $\frac{1}{b}\cdot D.$
\end{proof}

\section{Proof of the main results}
\label{sec:5}
\setcounter{equation}{0}

\begin{proof}[Proof of Theorem~\ref{th:3.1}]
The strategy is proving the estimate (\ref{3.1}) for the simplified gradient $\BF$ in place of $\nabla \Bu$ and then using the fact that $\BF$ and $\nabla\Bu$ are close to the order of $h,$ return to (\ref{3.1}). In the sequel the norm $\|\cdot\|$ will be the $L^2$ norm $\|\cdot\|_{L^2(\Omega)}$ unless specified otherwise. 
We prove the estimate (\ref{3.1}) block by block by freezing each of the variables $t,$ $\Gth$ and $z$ and considering the appropriate inequality on the $t,\Gth,z=$const cross sections of $\Omega.$ The following lemma is a crucial tool in the estimation of the off-diagonal entries of the blocks $t\Gth$ and $tz.$
\begin{lemma}
\label{lem:5.1}
Let $h,b>0$ with $0<h<\frac{b}{4}$ and assume the Lipschitz functions $\varphi_1,\varphi_2\colon[0,b]\to(0,\infty)$ satisfy the conditions 
\begin{equation}
\label{5.1}
h\leq \varphi_i(y)\leq C_1h,\quad |\nabla\varphi_i(y)|\leq C_2h,\quad\text{for all}\quad y\in[0,b], i=1,2.
\end{equation}
Denote the thin two dimensional domain $D=\{(x,y) \ : \  y\in(0,b), x\in (-\varphi_1(y), \varphi_2(y))\}.$ Given a displacement 
$\BU=(u(x,y),v(x,y))\in H^1(D,\mathbb R^2),$ the vector fields $\BGa,\BGb\in W^{1,\infty}(D,\mathbb R^2)$ and the function $w\in H^1(D,\mathbb R),$ denote the perturbed gradient as follows:
\begin{equation}
\label{5.2}
\BM=
\begin{bmatrix}
u_{x} & u_{y}+\BGa\cdot\BU\\
v_{x} & v_{y}+\BGb\cdot\BU+w
\end{bmatrix}.
\end{equation}
Assume $\epsilon\in (0,1),$ then the following Korn-like interpolation inequality holds:
\begin{equation}
\label{5.3}
\|\BM\|^2\leq C\left(\frac{\|u\|\cdot \|e(\BM)\|}{h}+\|e(\BM)\|^2+\left(\frac{1}{\epsilon}+h^2\right)\|\BU\|^2+(\epsilon+h^2)(\|w\|^2+\|w_x\|^2)\right),
\end{equation}
for all $h$ small enough. Here $C$ depends only on the quantities $b,$ $\|\BGa\|_{W^{1,\infty}},$ $\|\BGb\|_{W^{1,\infty}}$ and the norm $\|\cdot\|$ is the $L^2$ norm $\|\cdot\|_{L^2(D)}.$
\end{lemma}

\begin{proof}
Let us point out that in the proof of Lemma~\ref{lem:5.1}, the constant $C$ may depend only on $b,$ $\|\BGa\|_{W^{1,\infty}}$ $\|\BGb\|_{W^{1,\infty}}$ as well as the norm $\|\cdot\|$ will be $\|\cdot\|_{L^2(D)}.$ First of all, we can assume by density that $\BU\in C^2(\bar D).$ For functions $f,g\in H^1(D,\mathbb R)$ denote by $\BM_{f,g}$ the matrix
\begin{equation}
\label{5.4}
\BM_{f,g}=
\begin{bmatrix}
u_{x} & u_{y}+f\\
v_{x} & v_{y}+g
\end{bmatrix}.
\end{equation}
Assume $\tilde u(x,y)$ is the harmonic part of $u$ in $D,$ i.e., it is the unique solution of the Dirichlet boundary value problem
\begin{equation}
\label{5.5}
\begin{cases}
\Delta \tilde u(x,y)=0, & (x,y)\in R\\
\tilde u(x,y)=u(x,y), & (x,y)\in \partial R.
\end{cases}
\end{equation}
Note first that due to the fact that $u-\tilde u$ vanishes on the lateral boundary of $D$ and $D$ has a thickness of order $h,$ then we have by the Poincar\'e inequality in the horizontal direction, that
\begin{equation}
\label{5.6}
\|u-\tilde u\|\leq Ch\|\nabla(u-\tilde u)\|.
\end{equation}
Multiplying the identity 
$$\Delta (u-\tilde u)=u_{xx}+u_{yy}=(e_{11}(\BM_{f,g})-e_{22}(\BM_{f,g}))_{x}+(2e_{12}(\BM_{f,g}))_{y}+g_x-f_y$$
by $u-\tilde u$ and integrating by parts over $D$ we arrive at
$$
\int_{D}|\nabla(u-\tilde u)|^2=\int_{D}\left((u-\tilde u)_{x}(e_{11}(\BM_{f,g})-e_{22}(\BM_{f,g}))+
2(u-\tilde u)_{y}e_{12}(\BM_{f,g})+(f_y-g_x)(u-\tilde u)\right),
$$
which gives by the Schwartz inequality and by virtue of (\ref{5.6}), the bound
\begin{equation}
\label{5.7}
\|\nabla(u-\tilde u)\|\leq C\left[\|e(\BM_{f,g})\|+h(\|f_y\|+\|g_x\|)\right].
\end{equation}
Combining (\ref{5.6}) and (\ref{5.7}) we obtain
\begin{align}
\label{5.8}
\|\nabla(u-\tilde u)\|&\leq C\left[\|e(\BM_{f,g})\|+h(\|f_y\|+\|g_x\|)\right],\\ \nonumber
\|u-\tilde u\|&\leq Ch\left[\|e(\BM_{f,g})\|+h(\|f_y\|+\|g_x\|)\right].
\end{align}
In the next step we utilize the fact that $\tilde u$ is harmonic, thus we can apply the key estimate (\ref{4.32}) to it. Indeed, have by virtue of (\ref{4.32}) and the triangle inequality, that
\begin{align}
\label{5.9}
\|u_{y}+f\|^2&\leq 4(\|u_{y}-\tilde u_{y}\|^2+\|\tilde u_{y}\|^2+\|f\|^2)\\ \nonumber
&\leq C\left(\|\nabla(u-\tilde u)\|^2+\frac{1}{h}\|\tilde u\|\cdot\|\tilde u_{x}\|+\|\tilde u\|^2+\|\tilde u_{x}\|^2+\|f\|^2\right)\\ \nonumber
&\leq C\left(\|\nabla(u-\tilde u)\|^2+\frac{1}{h}(\|u\|+\|u-\tilde u\|)(\|u_{x}\|+\|\nabla(u-\tilde u)\|)\right)\\ \nonumber
&+C\left(\|u\|^2+\|u-\tilde u\|^2+\|u_{x}\|^2+\|\nabla(u-\tilde u)\|^2+\|f\|^2\right).
\end{align}
Taking into account the fact that $u_x$ is an entry of $e(\BM_{f,g})$ as well as the bounds (\ref{5.8}), it is easy to see that (\ref{5.9})
yields the estimate
\begin{align}
\label{5.10}
\|u_{y}+f\|^2&\leq C\left(\frac{1}{h}\|u\|\cdot\|e(\BM_{f,g})\|+\|u\|(\|f_y\|+\|g_x\|)+h^2(\|f_y\|^2+\|g_x\|^2)\right)\\ \nonumber
&+C(\|u\|^2+\|e(\BM_{f,g})\|^2+\|f\|^2).
\end{align}
Next we have for the special case $f=\BGa\cdot\BU$ and $g=\BGb\cdot\BU+w$ the obvious bounds
\begin{align}
\label{5.11}
\|f_y\|&\leq C\|U\|_{H^1(R)}\leq C(\|\BM_{f,g}\|+\|\BU\|+\|w\|),\\ \nonumber
\|g_x\|&\leq C\|U\|_{H^1(R)}+\|w_x\|\leq C(\|\BM_{f,g}\|+\|\BU\|+\|w_x\|).
\end{align}
Consequently, we assume $\epsilon>0$ is a parameter and estimate the summand $\|u\|(\|f_y\|+\|g_x\|)$ on the right hand side of (\ref{5.10}) by the Cauchy inequality as
$$\|u\|(\|f_y\|+\|g_x\|)\leq \frac{1}{\epsilon}\|u\|^2+\epsilon(\|f_y\|+\|g_x\|)^2,$$
and then utilize (\ref{5.11}) to derive (\ref{5.3}) from (\ref{5.10}). The proof of the lemma is complete.
\end{proof}

\textbf{The block $13$.} For the block $13$ we freeze the variable $\Gth$ and deal with two-variable functions in the $\Gth$=const cross sections of $\Omega$ which are two dimensional thin domains. We aim to prove that for any $\epsilon>0$ for small enough $h$ the estimate holds:
\begin{equation}
\label{5.12}
\|F_{13}\|^2+\|F_{31}\|^2\leq C\left(\frac{\|u_t\|\cdot\|e(\BF)\|}{h}+\|e(\BF)\|^2+\frac{1}{\epsilon}\|\Bu\|^2+\epsilon \|F_{12}\|^2\right),
\end{equation}
where the norms are over the whole domain $\Omega.$
\begin{proof}
Fix $\Gth\in (0,\omega)$ and consider the displacement $\BU=(u_t,A_zu_z)$ with the vector fields $\BGa=(0,-A_z\Gk_z),$ $\BGb=(A_z^2\Gk_z,-A_{z,z})$ and the function $w=\frac{A_zA_{z,\Gth}}{A_\Gth}u_\Gth$ in the variables $t$ and $z$ over the cross section
$D=\{(t,z) \ : \ z\in (z^1(\Gth),z^2(\Gth)), t\in (-g_1^h(\Gth,z), g_2^h(\Gth,z))\}$ to prepare an application of Lemma~\ref{lem:5.1}. We have that
$$\BM=
\begin{bmatrix}
u_{t,t} & u_{t,z}-A_z\Gk_z u_z\\[3ex]
A_zu_{z,t} & A_zu_{z,z}+A_z^2\Gk_z u_t+\frac{A_zA_{z,\Gth}}{A_\Gth}u_\Gth,
\end{bmatrix},
$$
thus (\ref{5.3}) written for the above choice of $\BU,$ $\BGa,\BGb$ and $w$ and then integrated in $\Gth$ over $(0,\omega)$ yields (\ref{5.12}).
\end{proof}
\textbf{The block $12$.} The role of the variables $\Gth$ and $z$ is the completely the same, thus we have an analogous estimate
\begin{equation}
\label{5.13}
\|F_{12}\|^2+\|F_{21}\|^2\leq C\left(\frac{\|u_t\|\cdot\|e(\BF)\|}{h}+\|e(\BF)\|^2+\frac{1}{\epsilon}\|\Bu\|^2+\epsilon \|F_{31}\|^2\right).
\end{equation}
Consequently adding (\ref{5.12}) and (\ref{5.13}) and choosing the parameter $\epsilon>0$ small enough we discover
\begin{equation}
\label{5.14}
\|F_{12}\|^2+\|F_{21}\|^2+\|F_{13}\|^2+\|F_{31}\|^2\leq C\left(\frac{\|u_t\|\cdot\|e(\BF)\|}{h}+\|e(\BF)\|^2+\|\Bu\|^2\right).
\end{equation}

\textbf{The block $23$.} The analysis for this block is a bit more technical, the complications arising due to the fact that $\Omega$ does not have constant thickness and thus one can not pursue exactly the same path as for the previous blocks, because the cross section $t=2h$ for instance may not even be connected. What comes to rescue is again the localization argument in [\ref{bib:Koh.Vog.}]. We divide the proof into two steps. We first prove an appropriate estimate on the part of $\Omega$ bounded by the surfaces $t=-h$ and $t=h,$ then we extent that estimate in the normal direction to entire $\Omega.$ Denote
\begin{equation}
\label{5.15}
\Omega^h=\{(t,\Gth,z)\in\Omega \ : \ t\in (-h,h)\}.
\end{equation}
\textbf{Step 1. An estimate in $\Omega^h.$} \textit{The following estimate holds:}
\begin{equation}
\label{5.16}
\|F_{23}\|_{L^2(\Omega^h)}^2+\|F_{32}\|_{L^2(\Omega^h)}^2\leq C(\|\Bu\|_{L^2(\Omega^h)}^2+\|e(\BF)\|_{L^2(\Omega^h)}^2).
\end{equation}
Let us make the following observation: \textit{Denote $D_t=\{(\Gth,z) \ : \  \Gth\in(0,\omega), z\in(z^1(\Gth),z^2(\Gth))\}.$ Assume $\varphi=\varphi(\Gth,z)\in C^1(D_t,\mathbb R)$ satisfies the conditions
\begin{equation}
\label{5.17}
0<c_1\leq \varphi(\Gth,z)\leq c_2,\quad \|\nabla \varphi(\Gth,z)\|\leq c_3,\quad\text{for all}\quad (\Gth,z)\in D_t.
\end{equation}
For a displacement $\BU=(u,v)\in H^1(D_t,\mathbb R^2),$ denote the matrix
\begin{equation}
\label{5.18} \BM_\varphi=
\begin{bmatrix}
u_x & \varphi u_y\\
v_x & \varphi v_y
\end{bmatrix}.
\end{equation}
Then there exists a constant $c>0,$ depending only on the constants $\omega,l,L,Z$ and $c_i,\ i=1,2,3,$ such that
\begin{equation}
\label{5.19}
\|\BM_\varphi\|_{L^2(D_t)}^2\leq c(\|e(\BM_\varphi)\|_{L^2(D_t)}^2+\|u\|_{L^2(D_t)}^2+\|v\|_{L^2(D_t)}^2).
\end{equation}
}
\begin{proof}
Considering the auxiliary vector field $\BW=(u,\frac{1}{\varphi}v)\colon D_t\to\mathbb R^2,$ we have that
\begin{equation}
\label{5.20}
 \nabla \BW=
\begin{bmatrix}
u_x & u_y \\
\frac{1}{\varphi}v_x-\frac{\varphi_x}{\varphi^2}v & \frac{1}{\varphi}v_y-\frac{\varphi_y}{\varphi^2}v
\end{bmatrix},
\end{equation}
and Korn's second inequality [\ref{bib:Kon.Ole.2}, Theorem~2] gives
\begin{equation}
\label{5.21}
\|\nabla \BW\|_{L^2(D_t)}^2\leq C(\|e(\BW)\|_{L^2(D_t)}^2+\|u\|_{L^2(D_t)}^2+\|v\|_{L^2(D_t)}^2),
\end{equation}
where the constant $C$ depends only on $\omega,l,L,Z$ and $c_i,\ i=1,2,3.$ It is then clear that (\ref{5.21}) bounds the $L^2(D_t)$ norms of the partial derivatives $u_y$ and $v_x$ by that of $u_x,$ $v_y,$ $u,$ $v$ and the sum $\varphi u_y+v_x$ by the triangle inequality, 
which is basically what is claimed in (\ref{5.19}).
\end{proof}
Returning back to (\ref{5.16}), we consider the function $\varphi(\Gth,z)=\frac{A_\Gth}{A_z}$ and apply (\ref{5.19}) to the displacement field
$\BU=(u_\Gth,u_z),$ to arrive at the estimate
\begin{align*}
\|u_{\Gth,z}\|&_{L^2(D_t)}^2+\|u_{z,\Gth}\|_{L^2(D_t)}^2\\ 
&\leq C\left(\|u_{\Gth,\Gth}\|_{L^2(D_t)}^2+\|u_{z,z}\|_{L^2(D_t)}^2
+\left\|\frac{A_\Gth}{A_z}u_{\Gth,z}+u_{z,\Gth}\right\|_{L^2(D_t)}^2+\|u_{\Gth}\|_{L^2(D_t)}^2+\|u_{z}\|_{L^2(D_t)}^2\right),
\end{align*}
which gives
\begin{align}
\label{5.22}
\|u_{\Gth,z}\|&_{L^2(D_t)}^2+\|u_{z,\Gth}\|_{L^2(D_t)}^2\\ \nonumber
&\leq C(\|u_{\Gth,\Gth}\|_{L^2(D_t)}^2+\|u_{z,z}\|_{L^2(D_t)}^2+\|A_\Gth u_{\Gth,z}+A_zu_{z,\Gth}\|_{L^2(D_t)}^2+
\|u_{\Gth}\|_{L^2(D_t)}^2+\|u_{z}\|_{L^2(D_t)}^2).
\end{align}
As the norms $\int_{R_t}f^2$ and $\int_{R_t}A_\Gth A_zf^2$ are equivalent, it is clear that (\ref{5.22}) implies (\ref{5.16}) by applying the triangle inequality several times to it and integrating the obtained estimate in $t\in(-h,h).$ \\

Before starting the second step, note that by combining the estimates (\ref{5.14}) and (\ref{5.16}) we get the bound 
\begin{equation}
\label{5.23}
\|\BF\|_{L^2(\Omega^h)}^2 \leq C\left(\frac{\|u_t\|\cdot\|e(\BF)\|}{h}+\|e(\BF)\|^2+\|\Bu\|^2\right).
\end{equation}
It is easy to see, that by an application of the obvious bounds
\begin{equation}
\label{5.24}
\|\BF-\nabla\Bu\|\leq h\|\nabla\Bu\|,\quad\text{and}\quad \|e(\BF)-e(\Bu)\|\leq h\|\nabla\Bu\|,
\end{equation}
we obtain from (\ref{5.23}) the partial estimate 
\begin{equation}
\label{5.25}
\|\nabla\Bu\|_{L^2(\Omega^h)}^2 \leq C\left(\frac{\|\Bu\cdot \Bn\|\cdot\|e(\Bu)\|}{h}+\|e(\Bu)\|^2+\|\Bu\|^2\right),
\end{equation}
for small enough $h.$ Next we extend the estimate (\ref{5.25}) in the normal direction to the entire $\Omega.$ The following simple lemma will be the key tool in the sequel. 
\begin{lemma}
\label{lem:5.2}
Assume $D_1\subset D_2\subset \mathbb R^n$ are open bounded connected Lipschitz domains. By Korn's first inequality, there exist constants $K_1$ and $K_2$ such that for any vector field $\BU\in H^1(D_2,\mathbb R^n),$ there exist skew-symmetric matrices $\BA_1,\BA_2\in\mathbb M^{n\times n},$ such that 
\begin{equation}
\label{5.26}
\|\nabla\BU-\BA_1\|_{L^2(D_1)} \leq K_1\|e(\BU)\|_{L^2(D_1)},\quad \|\nabla\BU-\BA_2\|_{L^2(D_2)} \leq K_2\|e(\BU)\|_{L^2(D_2)}.
\end{equation}
The assertion is that there exists a constant $C>0$ depending only on the quantities $K_1,K_2$ and $\frac{|D_2|}{|D_1|},$ such that for any 
vector field $\BU\in H^1(D_2, \mathbb R^n)$ one has
\begin{equation}
\label{5.27}
\|\nabla\BU\|_{L^2(D_2)} \leq C(\|\nabla\BU\|_{L^2(D_1)}+\|e(\BU)\|_{L^2(D_2)}).
\end{equation}
\end{lemma}

\begin{proof}
The proof is done owing to (\ref{5.26}) and the triangle inequality. We have first by the triangle inequality that 
\begin{equation}
\label{5.28}
\|\BA_1-\BA_2\|_{L^2(D_1)}\leq \|\nabla\BU-\BA_1\|_{L^2(D_1)}+\|\nabla\BU-\BA_2\|_{L^2(D_2)}\leq (K_1+K_2)\|e(\BU)\|_{L^2(D_2)}.
\end{equation}
Consequently, we have owing to (\ref{5.26}), (\ref{5.28}) and by several applications of the triangle inequality 
\begin{align*}
\|\nabla&\BU\|_{L^2(D_2)} \leq \|\nabla\BU-\BA_2\|_{L^2(D_2)}+\|\BA_2\|_{L^2(D_2)}\\
&\leq K_2\|e(\BU)\|_{L^2(D_2)}+\left(\frac{|D_2|}{|D_1|}\right)^{1/2} \left(\|\BA_2-\BA_1\|_{L^2(D_1)}+\|\BA_1\|_{L^2(D_1)}\right)\\
& \leq \left(K_2+(K_1+K_2)\left(\frac{|D_2|}{|D_1|}\right)^{1/2} \right) \|e(\BU)\|_{L^2(D_2)}+
\left(\frac{|D_2|}{|D_1|}\right)^{1/2}(\|\BA_1-\nabla\BU\|_{L^2(D_1)}+\|\nabla\BU\|_{L^2(D_1)})\\
&\leq C(\|\nabla\BU\|_{L^2(D_1)}+\|e(\BU)\|_{L^2(D_2)}),
\end{align*}
which is (\ref{5.27}).
\end{proof}

Assume now $\bar\Bu=(\bar u_1,\bar u_2, \bar u_3)$ is $\Bu$ in Cartesian coordinates $\Bx=(x_1,x_2,x_3)$ and denote by $\bar\nabla$ the cartesian gradient. we divide the domains $\Omega$ and $\Omega^h$ into small pieces of order $h.$ Namely for $N=[\frac{1}{h}]+1$ denote 
\begin{align}
\label{5.29}
&\Omega_{i,j}=\left\{(t,\Gth,z)\in\Omega \ : \ \Gth\in \left(\frac{i}{N},\frac{i+1}{N}\right), z\in \left(\frac{Nz_1+j(z_2-z_1)}{N}, \frac{Nz_1+(j+1)(z_2-z_1)}{N}\right)\right\},\\ \nonumber
&\Omega_{i,j}^h=\{(t,\Gth,z)\in\Omega_{i,j} \ : \ t\in (-h,h)\} ,\quad   i,j=0,1,\dots,N-1.
\end{align}
Note that as Korn's first inequality is invariant under the variable change $x\to\lambda x,$ then so is the constant $C$ in (\ref{5.27}). Second, the domains $D_1=\Omega_{i,j}^h$ and $D_2=\Omega_{i,j}^h$ have uniform Lipschitz constants depending only on the parameters mid-surface $S$ and the functions $g_1^h, g_2^h,$ are of order $h,$ thus by the above remark and Lemma~\ref{lem:5.2} we have the estimate
\begin{equation}
\label{5.30}
\|\bar\nabla\bar\Bu\|_{L^2(\Omega_{ij})} \leq C(\|\bar\nabla\bar\Bu\|_{L^2(\Omega_{ij}^h)}+\|e(\bar\Bu)\|_{L^2(\Omega_{ij})}),\quad   i,j=0,1,\dots,N-1,
\end{equation} 
summing which over $i,j=0,1,\dots,N-1$ we arrive at 
\begin{equation}
\label{5.31}
\|\bar\nabla\bar\Bu\|_{L^2(\Omega)} \leq C(\|\bar\nabla\bar\Bu\|_{L^2(\Omega^h)}+\|e(\bar\Bu)\|_{L^2(\Omega)}).
\end{equation} 
It remains to notice that (\ref{5.25}) and the analogous estimate with $\Bu$ and $\nabla$ replaced by $\bar\Bu$ and $\bar\nabla$ respectively are equivalent, thus (\ref{5.25}) and (\ref{5.31}) yield (\ref{3.1}). The Ansatz proving the sharpness of (\ref{3.1}) and (\ref{3.2}) has been constructed in [\ref{bib:Harutyunyan.1}] to prove the sharpness of the appropriate Korn's first inequality for shells with positive Gaussian curvature, and it turns out to work also for the estimates (\ref{3.1}) and (\ref{3.2}). We will omit the additional calculation here to prove that the Ansatz has the requited property. It is given as
\begin{equation}
\label{6.1}
\begin{cases}
u_t=W(\frac{\Gth}{\sqrt{h}},z)\\
u_\Gth=-\frac{t\cdot W_{,\Gth}\left(\frac{\Gth}{\sqrt h},z\right)}{A_\Gth{\sqrt h}}\\
u_z=-\frac{t\cdot W_{,z}\left(\frac{\Gth}{\sqrt h},z\right)}{A_z},
\end{cases}
\end{equation}
where $W(\xi,\eta)\colon\mathbb R^2\to\mathbb R$ is a fixed smooth compactly supported function.

\end{proof}

\section{Some remarks}
The works [\ref{bib:Gra.Har.3}] and [\ref{bib:Harutyunyan.2}] deal with Korn's first inequality for shells (when $\Omega$ has a constant thickness) in the case when the mid-surface $S$ consists of a single patch $\Sigma.$ It has been proven in [\ref{bib:Gra.Har.3}], that if one of the principal curvatures of $S$ is zero and the other one has a constant sign, then the optimal constant $C_{II}$ in (\ref{1.2}) scales like $Ch^{-3/2}$ for vector fields $\Bu\in H^1(\Omega)$ that satisfy zero boundary condition on the thin face of $\Omega.$ In that case one can of course also choose $\BA=0$ in (\ref{1.2}). The work [\ref{bib:Harutyunyan.2}] deals with shells that have nonzero Gaussian curvatures, where it is proven that $C_{II}$ scales like $Ch^{-4/3}$ for shells with negative Gaussian curvature and like $Ch^{-1}$ for shells with positive Gaussian curvature. The vector field $\Bu\in H^1(\Omega)$ again satisfies zero boundary conditions on the thin face of $\Omega,$ and there is an additional condition $L<C_0$ for some constant $C_0$ depending only the parameters of $S.$ First of all these two results can be extended to thin domains by the localization argument. Second, for the purpose of Korn's first inequality with no boundary conditions (\ref{1.2}), it has been proven in [\ref{bib:Harutyunyan.2}], that shells with nonzero Gaussian curvature $K_G$ fulfill the inequalities
\begin{itemize}
\item[(i)] If $K_G>0$ then $\|\Bu_{out}\|\leq C(\|\Bu_{in}\|+\|e(\Bu)\|),$ where $\Bu_{in}=0$ on the thin face of $\Omega.$ Here $\Bu_{out}=\Bn\cdot\Bu$ and 
$\Bu_{in}=\Bu-\Bu_{out}$ are the out-of-plane and in-plane components of the field $\Bu$ respectively. 
\item[(ii)] If $K_G<0,$ then $\|\Bu_{out}\|\leq \frac{C}{h^{1/3}}(\|\Bu_{in}\|+\|e(\Bu)\|),$ where $\Bu_{in}=0$ on the thin face of $\Omega.$ 
\end{itemize}
By a division of $S$ into single patches $\Sigma$ that satisfy the condition $L<C_0,$ we get applying the estimate (\ref{3.1}) to the vector field 
$\Bv=(\varphi u_t,\varphi u_\Gth,\varphi u_z)$ (here $\varphi$ is a cut-off function in both $\Gth$ and $z$ directions) the corresponding estimates 
\begin{itemize}
\item[(i)] If $K_G>0$ then $\|\nabla\Bu\|^2\leq \frac{C}{h}(\|\Bu_{in}\|^2+\|e(\Bu)\|^2)$ for all $\Bu\in H^1(\Omega).$ 
\item[(ii)] If $K_G<0,$ then $\|\nabla\Bu\|^2\leq \frac{C}{h^{4/3}}(\|\Bu_{in}\|^2+\|e(\Bu)\|^2)$ for all $\Bu\in H^1(\Omega).$ 
\end{itemize} 
The above estimates are believed to be useful in pursuing (\ref{2.1}) as it is known that the out-of-plane component of $\Bu$ is always harder to bound than the in-plane component. We also conjecture that the optimal constants $C_I$ and $C_{II}$ have the same asymptotics for thin domains $\Omega$ as the thickness $h$ goes to zero. While the estimate $C_I\leq C_{II}$ is quite straightforward by considering the field $\Bv=\BI+\epsilon\Bu$ in (\ref{1.1}) and then 
taking $\epsilon\to 0,$ the reverse inequality $C_{II}\leq cC_I$ where $c$ does not depend on $\Bu$ and $h$ is not trivial and is task for future work.

\end{document}

%% file: mymacros.tex
\usepackage{amssymb,bm}
\usepackage{amsmath}
\usepackage{amsthm}
\usepackage{graphicx}
\usepackage[active]{srcltx} % Includes line number info into dvi file
\usepackage{hyperref}
\hypersetup{pdfborder=0 0 0}

\newtheorem{theorem}{{\sc Theorem}}[section]

\newtheorem{lemma}[theorem]{{\sc Lemma}}

\def\XXint#1#2#3{{\setbox0=\hbox{$#1{#2#3}{\int}$ }
\vcenter{\hbox{$#2#3$ }}\kern-.6\wd0}}

\newcommand{\Gk}{\kappa}

\newcommand{\Gth}{\theta}

\bmdefine\BGa{\alpha}
\bmdefine\BGb{\beta}
\bmdefine\BGd{\delta}
\bmdefine\BGe{\epsilon}
\bmdefine\BGve{\varepsilon}
\bmdefine\BGf{\phi}
\bmdefine\BGvf{\varphi}
\bmdefine\BGg{\gamma}
\bmdefine\BGc{\chi}
\bmdefine\BGi{\iota}
\bmdefine\BGk{\kappa}
\bmdefine\BGl{\lambda}
\bmdefine\BGn{\eta}
\bmdefine\BGm{\mu}
\bmdefine\BGv{\nu}
\bmdefine\BGp{\pi}
\bmdefine\BGth{\theta}
\bmdefine\BGvth{\vartheta}
\bmdefine\BGr{\rho}
\bmdefine\BGvr{\varrho}
\bmdefine\BGs{\sigma}
\bmdefine\BGvs{\varsigma}
\bmdefine\BGt{\tau}
\bmdefine\BGj{\tau}
\bmdefine\BGu{\upsilon}
\bmdefine\BGo{\omega}
\bmdefine\BGx{\xi}
\bmdefine\BGy{\psi}
\bmdefine\BGz{\zeta}
\bmdefine\BGD{\Delta}
\bmdefine\BGF{\Phi}
\bmdefine\BGG{\Gamma}
\bmdefine\BGL{\Lambda}
\bmdefine\BGP{\Pi}
\bmdefine\BGT{\Theta}
\bmdefine\BGS{\Sigma}
\bmdefine\BGU{\Upsilon}
\bmdefine\BGO{\Omega}
\bmdefine\BGX{\Xi}
\bmdefine\BGY{\Psi}

\bmdefine\BCA{{\mathcal A}}
\bmdefine\BCB{{\mathcal B}}
\bmdefine\BCC{{\mathcal C}}
\bmdefine\BCD{{\mathcal D}}
\bmdefine\BCE{{\mathcal E}}
\bmdefine\BCF{{\mathcal F}}
\bmdefine\BCG{{\mathcal G}}
\bmdefine\BCH{{\mathcal H}}
\bmdefine\BCI{{\mathcal I}}
\bmdefine\BCJ{{\mathcal J}}
\bmdefine\BCK{{\mathcal K}}
\bmdefine\BCL{{\mathcal L}}
\bmdefine\BCM{{\mathcal M}}
\bmdefine\BCN{{\mathcal N}}
\bmdefine\BCO{{\mathcal O}}
\bmdefine\BCP{{\mathcal P}}
\bmdefine\BCQ{{\mathcal Q}}
\bmdefine\BCR{{\mathcal R}}
\bmdefine\BCS{{\mathcal S}}
\bmdefine\BCT{{\mathcal T}}
\bmdefine\BCU{{\mathcal U}}
\bmdefine\BCV{{\mathcal V}}
\bmdefine\BCW{{\mathcal W}}
\bmdefine\BCX{{\mathcal X}}
\bmdefine\BCY{{\mathcal Y}}
\bmdefine\BCZ{{\mathcal Z}}

\bmdefine\Bzr{ 0}
\bmdefine\Ba{ a}
\bmdefine\Bb{ b}
\bmdefine\Bc{ c}
\bmdefine\Bd{ d}
\bmdefine\Be{ e}
\bmdefine\Bf{ f}
\bmdefine\Bg{ g}
\bmdefine\Bh{ h}
\bmdefine\Bi{ i}
\bmdefine\Bj{ j}
\bmdefine\Bk{ k}
\bmdefine\Bl{ l}
\bmdefine\Bm{ m}
\bmdefine\Bn{ n}
\bmdefine\Bo{ o}
\bmdefine\Bp{ p}
\bmdefine\Bq{ q}
\bmdefine\Br{ r}
\bmdefine\Bs{ s}
\bmdefine\Bt{ t}
\bmdefine\Bu{ u}
\bmdefine\Bv{ v}
\bmdefine\Bw{ w}
\bmdefine\Bx{ x}
\bmdefine\By{ y}
\bmdefine\Bz{ z}
\bmdefine\BA{ A}
\bmdefine\BB{ B}
\bmdefine\BC{ C}
\bmdefine\BD{ D}
\bmdefine\BE{ E}
\bmdefine\BF{ F}
\bmdefine\BG{ G}
\bmdefine\BH{ H}
\bmdefine\BI{ I}
\bmdefine\BJ{ J}
\bmdefine\BK{ K}
\bmdefine\BL{ L}
\bmdefine\BM{ M}
\bmdefine\BN{ N}
\bmdefine\BO{ O}
\bmdefine\BP{ P}
\bmdefine\BQ{ Q}
\bmdefine\BR{ R}
\bmdefine\BS{ S}
\bmdefine\BT{ T}
\bmdefine\BU{ U}
\bmdefine\BV{ V}
\bmdefine\BW{ W}
\bmdefine\BX{ X}
\bmdefine\BY{ Y}
\bmdefine\BZ{ Z}